\documentclass[leqno]{amsart}
\pdfoutput=1
\usepackage[whole]{bxcjkjatype}
\usepackage[a4paper, margin=3cm]{geometry}
\usepackage{amsthm}
\usepackage{amsmath,amssymb,mathtools,mathdots}
\usepackage{mathrsfs}
\usepackage{url}
\usepackage[all]{xy}
\usepackage{tikz}
\usepackage{tikz-cd}
\usetikzlibrary {arrows,backgrounds}
\usetikzlibrary{shapes.geometric, intersections, calc}
\usetikzlibrary{decorations.pathmorphing, arrows}
\usepackage{comment}
\usepackage{placeins}

\newtheorem{theorem}{Theorem}[section]
\newtheorem*{theorem*}{Theorem}
\newtheorem{lemma}[theorem]{Lemma}
\newtheorem*{lemma*}{Lemma}
\newtheorem{proposition}[theorem]{Proposition}
\newtheorem*{proposition*}{Proposition}

\newtheorem*{corollary*}{Corollary}

\newtheorem*{claim*}{Claim}

\newtheorem*{fact*}{Fact}

\newtheorem*{conjecture*}{Conjecture}
\theoremstyle{definition}
\newtheorem{definition}[theorem]{Definition}
\newtheorem*{definition*}{Definition}
\newtheorem{example}[theorem]{Example}
\newtheorem*{example*}{Example}
\newtheorem{remark}[theorem]{Remark}
\newtheorem*{remark*}{Remark}

\newtheorem*{question*}{Question}

\newtheorem*{assumption*}{Assumption}
\numberwithin{equation}{section}
\allowdisplaybreaks[1]

\newcommand{\abs}[1]{\left\lvert#1\right\rvert}

\DeclareMathOperator{\Hom}{Hom}

\DeclareMathOperator{\id}{id}

\DeclareMathOperator{\ord}{ord}

\DeclareMathOperator{\Aut}{Aut}
\DeclareMathOperator{\pr}{pr}

\let\Im\relax
\DeclareMathOperator{\Im}{Im}
\newcommand{\C}{{\mathbb C}}
\newcommand{\R}{{\mathbb R}}
\newcommand{\Q}{{\mathbb Q}}
\newcommand{\Z}{{\mathbb Z}}

\newcommand{\Half}{{\mathbb H}}

\newcommand{\conj}[1]{{\overline{#1}}}

\newcommand{\OO}[1]{{\mathcal O_{#1}}}
\newcommand{\unit}[1]{{{\mathcal O}_{#1}^{\times}}}
\newcommand{\unitp}[1]{{{\mathcal O}_{#1}^{\times, +}}}
\usepackage{multirow}

\address{graduate school of mathematical sciences, the university of tokyo, 3-8-1 komaba, meguro, tokyo 153-8914, japan.}
\email{shuho@ms.u-tokyo.ac.jp}
\begin{document}
\title
[
Several complex structures on the Oeljeklaus-Toma manifolds
]
{
Several complex structures on the Oeljeklaus-Toma manifolds
}
\author{Shuho Kanda}
\date{}
\maketitle
\begin{abstract}
    We investigate complex structures on the Oeljeklaus-Toma manifolds. 
The Oeljeklaus-Toma manifolds are defined using complex embeddings of number fields. 
By replacing these embeddings with their conjugates, 
one obtains other manifolds that share the same underlying differential structure. 
In this paper, 
we give an algebraic description of when such manifolds are biholomorphic. 
As a simple application, 
we obtain compact non-K\"{a}hler manifolds of dimension $2t+1$ 
that admit $2^t$ different rigid complex structures. 
\end{abstract}

\setcounter{tocdepth}{1}
\tableofcontents

\section{Introduction}

\subsection{Background}
The Oeljeklaus-Toma (OT, for short) manifolds, 
introduced by Oeljeklaus and Toma in \cite{OT05}, 
form a class of compact non-K\"{a}hler manifolds 
including Inoue-Bombieri surfaces \cite{Ino74} 
of type $S_M$. 
They are constructed from arithmetic data, 
which allows various geometric properties to be encoded in algebraic conditions. 
Moreover, 
OT manifolds are constructed as quotients of solvable Lie groups by lattices 
\cite{Kas13b}, 
which makes computations of many geometric quantities relatively tractable. 
For these reasons, 
OT manifolds have played a central role in Hermitian geometry, 
particularly in the study of non-K\"{a}hler geometry. 
For various properties of OT manifolds, 
we refer to \cite{APV20}, \cite{Kas23}, \cite{OV11}.
For special Hermitian metrics admitted by OT manifolds, 
see \cite{ADOS24}, \cite{DV23}, \cite{FV24}, \cite{Oti22}.
In \cite{Kan25}, 
a certain characterization of OT manifolds 
within the framework of locally conformally K\"{a}hler geometry is given. 

In what follows, 
we describe the arithmetic data used in the construction of OT manifolds. 
Let $K$ be a number field of degree $n = [ K : \Q]$ 
and of signature $(s,t)$, 
that is, 
$K$ admits $s$ real embeddings 
$\sigma_1, \ldots ,\sigma_s : K \hookrightarrow \R$, 
and $2t$ complex embeddings
$\sigma_{s+1}, \ldots ,\sigma_{s+2t} : K \hookrightarrow \C$. 
Since the complex embeddings come in conjugate pairs, 
we can label the indices so that $\sigma_{s+i} = \conj{\sigma_{s+t+i}}$ 
for all $1 \le i \le t$. 
We always assume that $s,t \ge 1$. 
Taking a free subgroup $U \subset \unit{K}$ of rank $s$
that satisfies appropriate conditions 
(called \emph{admissible}, Definition \ref{Def:admissible}), 
we can define a compact complex manifold $X(K,U)$ of 
complex dimension $s+t$. 
We call this manifold an \emph{OT manifold} of type $(s,t)$. 
OT manifolds of type $(1,1)$ are known as Inoue-Bombieri surfaces. 

Although the notation $X(K, U)$ is used in all references, 
the construction in fact depends also on a labeling of the embeddings 
$\{ \sigma_1, \ldots, \sigma_{s+2t} \}$. 
For example, 
by permuting the embeddings $\sigma_{s+1}$ and $\sigma_{s+t+1}$, 
one obtains a complex manifold which is diffeomorphic 
but not necessarily biholomorphic to the original one. 
The additional data involved in the construction, 
apart from $K$ and $U$, 
essentially corresponds to elements of the following set: 
\[
    \mathcal{T}_{K} \coloneqq \left\{ T \subset \{ \sigma_{s+1}, \ldots ,\sigma_{s+2t} \} 
    \, \bigg| \,  | T \cap \{\sigma_{s+i},\sigma_{s+t+i}\} |=1
    \; \text{for all $1 \le i \le t$}
    \right\}, 
\]
which is determined uniquely by $K$. 
The elements of $\mathcal{T}_{K}$ represent 
the complex embeddings used to construct 
an OT manifold from $(K,U)$. 
In this paper, 
we explicitly include the additional data $T \in \mathcal{T}_{K}$, 
and denote the OT manifold 
(more precisely, its biholomorphism class) by $X(K,U,T)$. 
Since $\rvert \mathcal{T}_{K} \lvert=2^t$, 
at most $2^t$ distinct complex manifolds 
can be obtained from the data $(K,U)$ when constructing OT manifolds. 
In the case of Inoue-Bombieri surfaces, 
that is, 
when $(s,t)=(1,1)$, 
it is shown in \cite{Ino75}, 
which is only available in Japanese, 
that these $2^t=2$ complex structures are not biholomorphic. 
It is further shown that the complex structures on 
an Inoue-Bombieri surface are limited to these two. 

\subsection{Result}
The purpose of this paper is to give an algebraic description 
of which of these $2^t$ complex structures are biholomorphic to each other. 
The result is described as follows: 

\begin{theorem}
[= Theorem \ref{Thm:main_theorem}]
    Let $K$ be a number field of signature $(s,t)$ and $U$ be 
    an admissible subgroup such that 
    $\Q(U)=K$. 
    Define a group 
    \[
    A_U \coloneq \left\{ \rho \in \Aut(K/\Q) \,\Big| \, \rho(U)=U \right\} 
    \]
    and an action $A_U \curvearrowright \mathcal{T}_{K}$ by 
    \[
    \rho \cdot T \coloneqq \{ \sigma \circ \rho \mid \sigma \in T \} 
    \]
    for $T \in \mathcal{T}_K$. 
    Then, 
    for $T, T' \in \mathcal{T}_{K}$, 
    we have $X(K,U,T)=X(K,U,T')$ 
    if and only if 
    there exists $\rho \in A_U$ such that $\rho \cdot T = T'$. 
    In particular, 
    there are $\lvert \mathcal{T}_{K} / A_U \rvert$
    biholomorphism classes of complex manifolds 
    in the set
    $\{ X(K,U,T) \mid T \in \mathcal{T}_{K} \}$. 
\end{theorem}

An OT manifold $X(K,U,T)$ which satisfies $\Q(U)=K$ is called 
\emph{of simple type}, 
as defined in \cite{OT05}. 
After Definition \ref{Def:of_simple_type}, 
it is explained that $X(K,U,T)$ is of simple type 
for generic admissible subgroups $U \subset \unitp{K}$. 

\subsection{Examples}
We give a few examples here. 
Let $K = \Q(\sqrt[m]{2})$ for an odd number $m=2t+1$. 
Then, 
$K$ has one real embedding and $2t$ complex embeddings. 
Since $\Aut(K/\Q) = \{ \id \}$, 
for all admissible subgroups $U$, 
we obtain $2^t$ OT manifolds $\{ X(K,U,T) \}_{T \in \mathcal{T}_{K}}$ 
which are not biholomorphic to each other. 
Furthermore, in \cite{APV20}, 
it is shown that these complex structures on 
OT manifolds of simple type are rigid. 
As a result, 
we obtain an OT manifold of dimension $2t+1$ 
that admits $2^t$ different rigid complex structures. 
In contrast, 
for instance, 
in the case of $K=\Q(\sqrt{1+\sqrt{2}})$ and $U=\unitp{K}$, 
we have $\lvert \mathcal{T}_{K} / A_U \rvert =1$. 
See Example \ref{Ex:(2,1)} and Example \ref{Ex:(1,t)}. 

\subsection{Strategy of the proof}
As suggested by the classical notation $X(K,U)$ for OT manifolds, 
most of their geometric properties do not reflect 
the data $\mathcal{T}_{K}$. 
However, 
Theorem \ref{Thm:cohomology_on_OT}, 
a result by Kasuya \cite{Kas23} 
about the Dolbeault cohomology of OT manifolds 
with values in a certain class of flat line bundles, 
illustrates how 
the geometric properties 
of OT manifolds reflect the data $\mathcal{T}_{K}$. 
In the proof of the main theorem, 
we construct nontrivial sections of certain flat line bundles
from a biholomorphism between two OT manifolds. 
By comparing the existence of these sections 
with Theorem \ref{Thm:cohomology_on_OT}, 
we obtain algebraic constraints. 
We note that a similar method was used in \cite{BV18} 
to determine the automorphism groups of Oeljeklaus–Toma manifolds.

\subsection{Organization of this paper}
The paper is organized as follows: 
\begin{itemize}
\item In Section \ref{section:Oeljeklaus-Toma manifolds}, 
we review the construction of OT manifolds 
and explain how the data $\mathcal{T}_K$ is involved in the construction.
Moreover, 
we refer to a result by Kasuya, 
which plays a crucial role in the proof of the main theorem. 
\item In Section 
\ref{section:Biholomorphic equivalence of conjugate OT manifolds}, 
we prove Theorem \ref{Thm:main_theorem}. 
\item In Section \ref{section:Examples}, 
we give some examples. 
\end{itemize}

\addtocontents{toc}{\protect\setcounter{tocdepth}{0}}
\section*{Acknowledgments}
\addtocontents{toc}{\protect\setcounter{tocdepth}{1}}
I would like to express my gratitude to my supervisor, 
Kengo Hirachi, for his enormous support. 
I am also grateful to Hisashi Kasuya 
for suggesting this interesting problem. 
I extend my sincere thanks to Junnosuke Koizumi 
for helpful discussions.  
He constructed Example \ref{Ex:(2,1)} prior to the completion of the main theorem, 
and this example eventually led to a concise formulation of the result. 
This research is supported by FoPM, WINGS Program, the University of Tokyo, 
and JSPS KAKENHI Grant number 24KJ0931. 

\section{Oeljeklaus-Toma manifolds}
\label{section:Oeljeklaus-Toma manifolds}

The Oeljeklaus-Toma manifolds,  
first defined in \cite{OT05}, 
are higher-dimensional generalizations of Inoue-Bombieri surfaces \cite{Ino74}. 
In this section, 
we review the definition and properties of Oeljeklaus–Toma manifolds, 
giving a more precise formulation than in previous references 
by explicitly specifying the complex embeddings involved.

\subsection{Construction}

Let $K$ be a number field of degree $n = [ K : \Q]$ 
and of signature $(s,t)$, 
that is, 
$K$ admits $s$ real embeddings 
$\sigma_1, \ldots ,\sigma_s : K \hookrightarrow \R$, 
and $2t$ complex embeddings
$\sigma_{s+1}, \ldots ,\sigma_{s+2t} : K \hookrightarrow \C$. 
Since the complex embeddings come in conjugate pairs, 
we can label the indices so that $\sigma_{s+i} = \conj{\sigma_{s+t+i}}$ 
for all $1 \le i \le t$. 
We always consider only the case where $s,t \ge 1$. 
Define the following set: 
\[
    \mathcal{T}_{K} \coloneqq \left\{ T \subset \{ \sigma_{s+1}, \ldots ,\sigma_{s+2t} \} 
    \, \bigg| \,  | T \cap \{\sigma_{s+i},\sigma_{s+t+i}\} |=1
    \; \text{for all $1 \le i \le t$}
    \right\}, 
\]
and $\widetilde{\mathcal{T}}_K$ to be the set of the elements 
of $\mathcal{T}_{K}$ endowed with an ordering, 
that is, 
\[
    \widetilde{\mathcal{T}}_K \coloneqq \left\{ 
    \tau=(\tau_{1},\ldots,\tau_{t}) \in \{ \sigma_{s+1}, \ldots ,\sigma_{s+2t} \}^t
    \, \bigg| \, \{ \tau_{1},\ldots,\tau_{t} \} \in \mathcal{T}_K
    \right\}. 
\]
We define a map $[ \cdot ] : \widetilde{\mathcal{T}}_K \to \mathcal{T}_K$ by 
$[(\tau_{1},\ldots,\tau_{t})]=\{ \tau_{1},\ldots,\tau_{t}\}$, 
which disregards the ordering. 
The sets $\mathcal{T}_K, \widetilde{\mathcal{T}}_K$ and the map $[\cdot]$ are 
uniquely determined by $K$. 

Fix an element $\tau=(\tau_{1},\ldots,\tau_{t}) \in \widetilde{\mathcal{T}}_K$. 
We construct a compact complex manifold of dimension $s+t$ from 
$K$, $\tau$ and an admissible subgroup $U \subset \unitp{K}$, 
which will be defined later. 

Let $\OO{K}$ denote the ring of algebraic integers of $K$, 
$\unit{K}$ the group of units in $\OO{K}$, 
and $\unitp{K}$ the group of \emph{totally positive units}, 
defined by 
\[
    \unitp{K} \coloneqq 
    \left\{ u \in \unit{K}\, \mid \, 
    \sigma_i(u)>0 \; \text{for all $1 \le i \le s$} \right\}. 
\]
Let $\Half \coloneqq \{ z \in \C\, \mid \, \Im z > 0 \}$ be 
the upper half-plane. 
Consider the action $\OO{K} \curvearrowright \Half^s \times \C^t$ given by translations: 
\[
    T_a(w_1, \ldots ,w_s,z_1, \dots ,z_t) \coloneqq 
    (w_1 + \sigma_1(a), \ldots ,w_s + \sigma_{s}(a), 
    z_1+\tau_1(a),\ldots,z_t+\tau_t(a))
\]
for $a \in \OO{K}$, 
and the action $\unitp{K} \curvearrowright \Half^s \times \C^t$ given by rotations: 
\[
    R_u(w_1, \ldots ,w_s,z_1, \dots ,z_t) \coloneqq 
    (\sigma_1(u) \cdot w_1, \ldots ,\sigma_{s}(u) \cdot w_s, 
    \tau_1(u)\cdot z_1,\ldots,\tau_t(u)\cdot z_t)
\]
for $u \in \unitp{K}$. 
These actions define an action 
$( \unitp{K} \ltimes \OO{K} ) \curvearrowright \Half^s \times \C^t$. 
The rank of the group $\OO{K}$ is $n$ since 
$\OO{K} \otimes_{\Z} \Q \simeq K$ holds. 
The group structure of $\unitp{K}$ is described by Dirichlet's unit theorem as follows: 

\begin{theorem}[Dirichlet's unit theorem]
\label{Thm:Dirichlet's unit theorem}
    The image of the following inclusion
    \[
    \begin{array}{c}
        \log : \unitp{K} \hookrightarrow \R^{s+t} \\
        u \mapsto (\log \sigma_1(u), \ldots ,\log \sigma_s(u),
        \log \abs{\tau_{1}(u)}^2, \ldots ,\log \abs{\tau_{t}(u)}^2)
    \end{array}
    \]
    is a lattice in the hyperplane 
    $H \coloneqq \left\{x \in \R^{s+t}\, \mid \, 
    \sum_{i=1}^{s+t} x_i =0  \right\}$. 
\end{theorem}

In particular, 
the rank of the group $\unitp{K}$ is $s+t-1$. 
To make the action 
$( \unitp{K} \ltimes \OO{K} ) \curvearrowright \Half^s \times \C^t$ discrete, 
we take a subgroup $U \subset \unitp{K}$ of lower rank. 

\begin{definition}\label{Def:admissible}
    Let $\pr_{\R^s} : \R^{s+t} \to \R^{s}$ be the projection onto the first $s$ coordinates. 
    We call a subgroup $U \subset \unitp{K}$ \emph{admissible} if 
    the rank of $U$ is $s$ and the image
    $\pr_{\R^s}(\log (U))$ is a lattice in $\R^s$. 
\end{definition}

By Dirichlet's unit theorem, 
one can always find an admissible subgroup $U$. 
We define a group $G_{\tau} \coloneqq U \ltimes \OO{K}$ and 
an action $G_{\tau}\curvearrowright (\Half^s \times \C^t)$ as 
the restriction of the action defined above. 
Oeljeklaus and Toma proved in \cite{OT05} that this action
is fixed-point-free, properly discontinuous, 
and co-compact for any 
admissible subgroup $U \subset \unitp{K}$. 
Now we can define OT manifolds. 

\begin{definition}
    An \emph{OT manifold} associated to a number field $K$, 
    an admissible subgroup $U \subset \unitp{K}$ and 
    $\tau \in \widetilde{\mathcal{T}}_K$ is 
    a compact complex manifold constructed as the quotient 
    \[
        X(K,U,\tau) \coloneqq (\Half^s \times \C^t) / G_{\tau}. 
    \]
    We call this an OT manifold \emph{of type $(s,t)$}. 
    For $\tau, \tau' \in \widetilde{\mathcal{T}}_K$, 
    we say the OT manifolds $X(K,U,\tau)$ and $X(K,U,\tau')$ are \emph{conjugate}. 
\end{definition}

\begin{definition}
    Let $X$ be a complex manifold. 
    Denote the biholomorphism class of $X$ by $[X]$. 
    For $T=[\tau] \in \mathcal{T}_K$, 
    we define
    \[
        X(K,U,T) \coloneqq [X(K,U,\tau)], 
    \]
    which is well-defined since a biholomorphism between
    $X(K,U,\tau)$ and $X(K,U,\tau')$ with $[\tau]=[\tau']$ 
    can be constructed 
    by a permutation of the coordinates of $\C^t$.
\end{definition}

\begin{remark}
    OT manifolds of type $(1,1)$ are known as Inoue-Bombieri surfaces. 
    OT manifolds originally arose from an attempt to generalize 
    Inoue-Bombieri surfaces to manifolds of higher-dimension. 
\end{remark}

By taking complex conjugates of some of the coordinates of $\C^t$, 
we can easily show that conjugate OT manifolds are diffeomorphic. 

\begin{proposition}
    For all $\tau, \tau' \in \widetilde{\mathcal{T}}_K$, 
    the complex manifolds $X(K,U,\tau)$ and $X(K,U,\tau')$ are 
    diffeomorphic. 
\end{proposition}

In \cite{OT05}, 
where OT manifolds are originally defined, 
certain OT manifolds are called \emph{of simple type}. 
The main theorem of this paper is precisely about 
OT manifolds of simple type. 

\begin{definition}\label{Def:of_simple_type}
    Let $(K,U)$ be a pair of a number field $K$ 
    and an admissible subgroup $U \subset \unitp{K}$. 
    We call $(K,U)$ \emph{of simple type} if 
    there exists no proper intermediate field extension 
    $\Q \subsetneq K' \subsetneq K$
    such that $U \subset \unitp{K'}$. 
    We call an OT manifold $X(K,U,\tau)$ \emph{of simple type} if 
    $(K,U)$ is of simple type. 
\end{definition}

There are at most finitely many 
proper intermediate field extensions 
$\Q \subsetneq K_1, \ldots, K_r \subsetneq K$. 
Since $\unitp{K_1}, \ldots, \unitp{K_r}$ are subgroups of $\unitp{K}$
of lower rank, 
$(K,U)$ is of simple type 
for generic admissible subgroups $U \subset \unitp{K}$. 
It can easily be seen that $(K,U)$ is of simple type 
if and only if there exists $u \in U$ such that $\Q(u)=K$. 
For other equivalent conditions, see \cite[Lemma 3.11]{Kan25}. 
We also note that an OT manifold of type 
$(s,1)$ is of simple type (\cite[Lemma 1.6]{OT05}). 

\subsection{The Dolbeault cohomology of OT manifolds}

We cite a result by Kasuya \cite{Kas23} 
on the Dolbeault cohomology of OT manifolds 
with values in a certain class of flat line bundles. 
As the result illustrates how 
the geometric properties of $X(K,U,\tau)$ reflect the data $\tau$, 
it plays a crucial role in the proof of the main theorem. 

We consider a flat line bundle $L_{\theta}$ over an OT manifold $X(K,U,\tau)$,  
defined by $\theta \in \Hom(U,\C^*)$. 
More precisely, 
$L_{\theta}$ is constructed as follows. 
We consider the action of $G_{\tau}$ on $(\Half^s \times \C^t) \times \C$ 
defined by $(u,a) \cdot (x,c) = ((u,a)\cdot x , \theta(u)c)$ for 
$(u,a) \in G_{\tau}$, $x \in \Half^s \times \C^t$, 
and $c \in \C$. 
The natural projection 
$((\Half^s \times \C^t) \times \C)/ G_{\tau} \to X(K,U,\tau)$
defines the flat line bundle $E_{\theta}$. 
In \cite{Kas23}, 
the vanishing and non-vanishing of the Dolbeault cohomology 
of these line bundles are studied, 
and the results are described in the following algebraic terms: 

\begin{theorem}[\cite{Kas23}, Corollary 5.2]\label{Thm:cohomology_on_OT}
    Let $\tau=(\tau_1, \ldots,\tau_t) \in \widetilde{\mathcal{T}}_{K}$. 
    Then we have $H^{p,q}(X(K,U,\tau), E_{\theta}) \neq \{ 0 \}$ 
    if and only if 
    there exist sets $I \subset \{1,\ldots,s\}, K,L \subset \{1,\ldots,t\}$ with 
    $|I|+|K|=p$ and $|L| \le q$ which satisfy
    \[
    \theta(u) = 
    \prod_{i \in I} \sigma_i(u)
    \prod_{j \in J} \tau_j(u)
    \prod_{k \in K} \conj{\tau}_k(u) 
    \]
    for all $u \in U$. 
\end{theorem}

In particular, 
we have $H^{1,0}(X(K,U,\tau), E_{\theta}) \neq \{ 0 \}$ if and only if 
$\theta \in \{\sigma_1,\ldots,\sigma_s,\tau_1,\ldots,\tau_t\}$ as maps 
with domain $U$. 

\section{Biholomorphic equivalence of conjugate OT manifolds} 
\label{section:Biholomorphic equivalence of conjugate OT manifolds}

In this section, 
we prove the main theorem of this paper. 

\begin{theorem}\label{Thm:main_theorem}
    Let $K$ be a number field and $U$ 
    an admissible subgroup such that 
    $(K,U)$ is of simple type. 
    Define a group 
    \[
    A_U \coloneq \left\{ \rho \in \Aut(K/\Q) \,\Big| \, \rho(U)=U \right\} 
    \]
    and an action $A_U \curvearrowright \mathcal{T}_{K}$ by 
    \[
    \rho \cdot T \coloneqq \{ \sigma \circ \rho \mid \sigma \in T \} 
    \]
    for $T \in \mathcal{T}_K$. 
    Then, 
    for $T, T' \in \mathcal{T}_{K}$, 
    we have $X(K,U,T)=X(K,U,T')$ 
    if and only if 
    there exists $\rho \in A_U$ such that $\rho \cdot T = T'$. 
    In particular, 
    there are $\lvert \mathcal{T}_{K} / A_U \rvert$
    biholomorphism classes of complex manifolds 
    in the set
    $\{ X(K,U,T) \mid T \in \mathcal{T}_{K} \}$. 
\end{theorem}

Before proving the theorem, 
we establish the following 

\begin{lemma}\label{Lem:hom_to_U}
    The map $\Hom(G_{\tau}, U) \to \Hom(U, U)$ induced by 
    the natural inclusion $U \hookrightarrow U \ltimes \OO{K} = G_{\tau}$ is 
    an isomorphism. 
\end{lemma}

\begin{proof}
    Since any map $\rho \in \Hom(G_{\tau},U)$ is trivial on 
    the commutator subgroup $[G_{\tau},G_{\tau}]$, 
    it is in particular trivial on 
    \[
    \left\{ (u,0)(1,a)(u^{-1},0)(1,-a)=(0,(u-1)a) 
    \mid u \in U, a \in \OO{K} \right\}
    \subset [G_{\tau},G_{\tau}]. 
    \]
    Fix $1 \neq u \in U$. 
    As the $\Q$-linear map 
    $\times (u-1) \colon \OO{K} \otimes_{\Z} \Q \to \OO{K} \otimes_{\Z} \Q$ 
    is invertible, 
    we have $[\OO{K} : (u-1)\OO{K}] < \infty$. 
    Since $\rho$ is trivial on $(0,(u-1)\OO{K}) \subset [G_{\tau},G_{\tau}]$, 
    which is a finite-index subgroup of $(0,\OO{K})$, 
    and $U$ is torsion-free, 
    it follows that $\rho$ is trivial on $(0,\OO{K})$. 
\end{proof}

\begin{proof}[Proof of Theorem \ref{Thm:main_theorem}]
    First, 
    we assume that 
    there exists $\rho \in A_U$ such that $\rho \cdot T =T'$. 
    We can take 
    $\tau=(\tau_1,\ldots,\tau_t)$ and 
    $\tau'=(\tau'_1, \ldots ,\tau'_t) 
    \in \widetilde{\mathcal{T}}_K$ 
    so that 
    $[\tau]=T$, $[\tau']=T'$ and 
    $\tau'_i \circ \rho = \tau_i$ 
    for all $1 \le i \le t$. 
    There exists a permutation $\Sigma \in \mathfrak{S}_s$ such that 
    $\sigma_i \circ \rho = \sigma_{\Sigma(i)}$ for all 
    $1\le i \le s$. 
    Define a biholomorphism
    \[
    \begin{array}{c}
    \widetilde{\Phi} : \Half^s \times \C^t \stackrel{\sim}{\longrightarrow} 
    \Half^s \times \C^t \\
    ((w_i)_{i=1}^s,(z_j)_{j=1}^t) \mapsto 
    ((w_{\Sigma(i)})_{i=1}^s,(z_j)_{j=1}^t). 
    \end{array}
    \]
    For an element $g=(u,a) \in G_{\tau}$ and 
    $g' =(\rho(u),\rho(a)) \in G_{\tau'}$, 
    we have
    \begin{align*}
        \widetilde{\Phi}(g \cdot (w,z)) &=
        \widetilde{\Phi}((\sigma_i(u)w_i+\sigma_i(a))_{i=1}^s,
        (\tau_j(u)z_j+\tau_j(a))_{j=1}^t) \\
        &= ((\sigma_{\Sigma(i)}(u)w_{\Sigma(i)}+\sigma_{\Sigma(i)}(a))_{i=1}^s,
        (\tau_j(u)z_j+\tau_j(a))_{j=1}^t) \\
        &= ((\sigma_i(\rho(u))w_{\Sigma(i)}+\sigma_i(\rho(a)))_{i=1}^s,
        (\tau'_j(\rho(u))z_j+\tau'_j(\rho(a)))_{j=1}^t) \\
        &= g' \cdot ((w_{\Sigma(i)})_{i=1}^s,(z_j)_{j=1}^t) \\
        &= g' \cdot \widetilde{\Phi}((w,z)). 
    \end{align*}
    Therefore, 
    the map $\widetilde{\Phi}$ induces the biholomorphism
    \[
    \Phi : (\Half^s \times \C^t)/G_{\tau} \stackrel{\sim}{\longrightarrow} 
    (\Half^s \times \C^t)/G_{\tau'}, 
    \]
    which shows $X(K,U,T)=X(K,U,T')$. 

    Conversely, 
    we assume that $X(K,U,T)=X(K,U,T')$. 
    There exists a biholomorphism
    \[
    \Phi : (\Half^s \times \C^t)/G_{\tau} \stackrel{\sim}{\longrightarrow} 
    (\Half^s \times \C^t)/G_{\tau'}, 
    \]
    which induces a biholomorphism
    \[
    \widetilde{\Phi} : \Half^s \times \C^t \stackrel{\sim}{\longrightarrow} 
    \Half^s \times \C^t. 
    \]
    For all $g \in G_{\tau}$, 
    there is a unique element 
    $\rho'(g) \in G_{\tau'}$ 
    such that
    \begin{equation}\label{Equ:tildePhi_compatibility}
        \widetilde{\Phi}(g\cdot(w,z))=\rho'(g) \cdot \widetilde{\Phi}((w,z)). 
    \end{equation}
    Obviously, 
    the map $\rho': G_{\tau} \to G_{\tau'}$ is a group homomorphism. 
    So we can define a homomorphism $\rho : U \to U$ 
    as the composition of the following maps: 
    \[
    U \stackrel{\iota}{\hookrightarrow} U \ltimes \OO{K} = G_{\tau}
    \stackrel{\rho'}{\longrightarrow} 
    G_{\tau'} = 
    U \ltimes \OO{K} \stackrel{\pr}{\twoheadrightarrow} U, 
    \]
    where the maps $\iota$ and $\pr$ are 
    the natural inclusion and surjection respectively. 
    By Lemma \ref{Lem:hom_to_U}, we have 
    $\pr \circ \rho'(u,a)=\rho(u)$
    for all $g=(u,a) \in G_{\tau}$. 
    
    Write 
    $\widetilde{\Phi}(w,z)=((\xi_i(w,z))_{i=1}^s,(\zeta_j(w,z)_{j=1}^t))$. 
    By equation (\ref{Equ:tildePhi_compatibility}), 
    for all $g=(u,a) \in G_{\tau}$ we have
    \[
    \left\{
    \begin{aligned}
    g^*\xi_i &= \sigma_i(\rho(u))\xi_i + \sigma_i(\rho''(g)) \\
    g^*\zeta_j &= \tau'_j(\rho(u))\zeta_j + \tau'_j(\rho''(g)) 
    \end{aligned}
    \right. 
    \]
    for all $1 \le i \le s$ and $1 \le j \le t$, 
    where $\rho'(g)=(\rho(u),\rho''(g))$. 
    By taking the exterior derivative, 
    we have
    \[
    \left\{
    \begin{aligned}
    g^*(d\xi_i) &= \sigma_i(\rho(u))d\xi_i \\
    g^*(d\zeta_j) &= \tau'_j(\rho(u))d\zeta_i. 
    \end{aligned}
    \right. 
    \]
    for all $1 \le i \le s$ and $1 \le j \le t$. 
    As a result, 
    we have non-trivial sections of holomorphic $1$-form 
    with values in flat line bundles as follows: 
    \[
    \left\{
    \begin{aligned}
    0 \neq d\xi_i &\in H^{1,0}(X(K,U,\tau), L_{\sigma_i \circ \rho}) \\
    0 \neq d\zeta_j &\in H^{1,0}(X(K,U,\tau), L_{\tau'_j \circ \rho}). 
    \end{aligned}
    \right. 
    \]
    By Theorem \ref{Thm:cohomology_on_OT}, 
    we have 
    $\{\sigma_1\circ \rho,\ldots,\sigma_s\circ \rho,\tau_1'\circ \rho,\ldots,\tau_t'\circ \rho\} 
    \subset \{\sigma_1,\ldots,\sigma_s,\tau_1,\ldots,\tau_t\}$ 
    as maps with domain $U$. 
    We now prove that these two sets are equal. 
    By the simplicity of $U$, 
    we can take $u_0 \in U$ such that $\Q(u_0)=K$. 
    If $\sigma_i \circ \rho=\sigma_j \circ \rho = \sigma_k$ on $U$
    for some $1\le i,j,k\le s+2t$, 
    we have
    \begin{equation}\label{Equ:sigma}
        \sigma_k(K)=\sigma_k(\Q(u_0))=\Q(\sigma_k(u_0))
        =\Q(\sigma_i(\rho(u_0))) = \sigma_i(\Q(\rho(u_0))). 
    \end{equation}
    Thus, 
    we have $[K:\Q]=[\Q(\rho(u_0)):\Q]$ which shows $\Q(\rho(u_0))=K$. 
    Since $\sigma_i(\rho(u_0))=\sigma_j(\rho(u_0))$, 
    we have $\sigma_i(x)=\sigma_j(x)$ for all $x \in K$. 
    Therefore, 
    there are $s+t$ maps in the set 
    $\{\sigma_1\circ \rho,\ldots,\sigma_s\circ \rho,
    \tau_1'\circ \rho,\ldots,\tau_t'\circ \rho\}$ and we have 
    $\{\sigma_1\circ \rho,\ldots,\sigma_s\circ \rho,
    \tau_1'\circ \rho,\ldots,\tau_t'\circ \rho\} 
    = \{\sigma_1,\ldots,\sigma_s,\tau_1,\ldots,\tau_t\}$. 
    As there are exactly $s$ real values in 
    $\{\sigma_1(\rho(u_0)),\ldots,\sigma_s(\rho(u_0)),
    \tau'_1(\rho(u_0)),\ldots,\tau'_t(\rho(u_0))\}$, 
    we have 
    \[
    \left\{
    \begin{aligned}
    &\{\sigma_1\circ \rho,\ldots,\sigma_s\circ \rho\} = 
    \{\sigma_1,\ldots,\sigma_s \} \\
    &\{\tau_1'\circ \rho,\ldots,\tau_t'\circ \rho\} 
    = \{\tau_1,\ldots,\tau_t\}. 
    \end{aligned}
    \right. 
    \]
    Take permutations $\Sigma \in \mathfrak{S}_s$ and 
    $\Pi \in \mathfrak{S}_t$ so that 
    $\sigma_i \circ \rho=\sigma_{\Sigma(i)}$, 
    $\tau'_{j} \circ \rho = \tau_{\Pi(j)}$ for all 
    $1\le i \le s$ and $1 \le j \le t$. 
    By the same argument as in equation (\ref{Equ:sigma}), 
    we have $\sigma_1(K)=\sigma_{\Sigma(1)}(K)$. 
    Define an automorphism $\tilde\rho \in \Aut(K/\Q)$ which makes 
    the following diagram commutative: 
    \[
    \begin{tikzcd}
        K \arrow[r,  "\tilde\rho"] \arrow[d, "\sigma_{\Sigma(1)}"'] 
        & K \arrow[d, "\sigma_1"] \\
        \sigma_{\Sigma(1)}(K) \arrow[r,  equal] 
        & \sigma_1(K). 
    \end{tikzcd}
    \]
    For all $u \in U$, 
    we have $\sigma_1(\rho(u)) =\sigma_{\Sigma(1)}(u)=\sigma_1(\tilde\rho(u))$, 
    and hence $\rho(u)=\tilde\rho(u)$. 
    Thus, 
    $\tilde\rho$ is an extension of $\rho$. 
    Since the permutation $\Sigma$ has a finite order, 
    there exists an integer $1\le m$ such that $\sigma_1 \circ \rho^m =\sigma_1$ on $U$. 
    In particular, we have $\rho(U)=U$. 
    Therefore, 
    we have $\tilde\rho(U) = U$ and hence $\tilde\rho \in A_U$. 
    Finally, we verify that $\tilde\rho \cdot T =T'$. 
    We have 
    \[
    \sigma_i(\tilde\rho(u_0))=\sigma_i(\rho(u_0))
    =\sigma_{\Sigma(i)}(u_0), \quad 
    \tau_j(\tilde\rho(u_0))=\tau_j(\rho(u_0))
    =\tau_{\Pi(j)}(u_0)
    \]
    for all 
    $1\le i \le s$ and $1 \le j \le t$. 
    Since $\Q(u_0)=K$, 
    we have 
    \[
    \sigma_i \circ\tilde\rho=\sigma_{\Sigma(i)}, \quad 
    \tau_j \circ \tilde\rho=\tau_{\Pi(j)}
    \]
    on $K$, 
    which completes the proof. 
\end{proof}

\section{Examples}
\label{section:Examples}

\begin{example}\label{Ex:(1,1)}
Let $X(K,U,\tau)$ be an OT manifold of type $(1,1)$. 
Then we have $\lvert \mathcal{T}_{K} / A_U \rvert=2$ as follows. 
The element $\rho \in A_U$ satisfies $\rho(u)=u$ or $u^{-1}$ 
for a generator $u$ of $U$. 
If $\rho(u)=u^{-1}$, 
then the conjugates of $u$ are $u$, $u^{-1}$, and $u'$ for some $u' \in K$. 
Since their product equals $\pm 1$, 
we must have $u' = \pm 1$, 
which is a contradiction. 
Thus, 
we have $A_U=\{ \id \}$. 
Therefore, 
we obtain $\lvert \mathcal{T}_{K} / A_U \rvert=2$. 
\end{example}

It is shown in \cite{Ino75} that an OT manifold of type $(1,1)$, 
viewed as a differentiable manifold, 
admits no complex structures other than these two. 
This naturally suggests the following 

\begin{question*}
    Is the number of complex structures 
    on an OT manifold $X(K,U,T)$ of simple type 
    limited to these $\lvert \mathcal{T}_{K} / A_U \rvert$?
\end{question*}

Since the proof in the case of type $(1,1)$ 
relies on the classification theory of complex surfaces, 
the same approach cannot be applied straightforwardly 
in higher-dimensions. 

\begin{example}\label{Ex:(2,1)}
    Let $K \coloneqq \Q(\sqrt{1+\sqrt{2}}) \simeq \Q[x]/(x^4-2x^2-1)$, 
    and $U \coloneqq \unitp{K}$. 
    $K$ is a number field of signature $(2,1)$ and $(K,U)$ is of simple type. 
    The real embeddings $\sigma_1, \sigma_2 \colon K \hookrightarrow \R$ and 
    the complex embeddings $\sigma_3, \sigma_4 \colon K \hookrightarrow \C$ are 
    defined by 
    \begin{align*}
        \sigma_1 \colon \sqrt{1+\sqrt{2}} \mapsto \sqrt{1+\sqrt{2}} \quad 
        &\sigma_2 \colon \sqrt{1+\sqrt{2}} \mapsto -\sqrt{1+\sqrt{2}} \\
        \sigma_3 \colon \sqrt{1+\sqrt{2}} \mapsto \sqrt{1-\sqrt{2}} \quad
        &\sigma_4 \colon \sqrt{1+\sqrt{2}} \mapsto -\sqrt{1-\sqrt{2}}. 
    \end{align*}
    Take $\rho \in \Aut(K/\Q)$ defined by $\rho =\sigma_2$. 
    Then, 
    we have 
    $\sigma_1 \circ \rho=\sigma_2$, 
    $\sigma_2 \circ \rho=\sigma_1$, 
    $\sigma_3 \circ \rho=\sigma_4$ and 
    $\sigma_4 \circ \rho=\sigma_3$. 
    Since $\rho(\unit{K})=\unit{K}$, 
    we have $\rho \in A_U$. 
    As $\rho$ acts by exchanging $\sigma_3$ and $\sigma_4$, 
    we obtain $\lvert \mathcal{T}_{K} / A_U \rvert=1$. 
\end{example}

\begin{example}\label{Ex:(1,t)}
    For any number field $K$ of signature $(1,t)$ 
    (for example, 
    $K = \Q(\sqrt[m]{2})$ where $m=2t+1$) 
    and any admissible subgroup $U$ such that $(K,U)$ is of simple type, 
    we have $\lvert \mathcal{T}_{K} / A_U \rvert=2^t$. 
    Indeed, 
    the field $K$ can be written as $\Q(\alpha)$ with $\alpha \in \R$. 
    Any automorphism $\rho \in \Aut(K/\Q)$
    must map $\alpha$ to a real conjugate, 
    but $\alpha$ is the only real conjugate. 
    Thus, 
    we have $\rho=\id$. 
\end{example}

\begin{example}\label{Ex:(2s'+1,1)}
    For any number field $K$ of signature $(2s'+1,1)$ 
    and any admissible subgroup $U$, 
    we have $\lvert \mathcal{T}_{K} / A_U \rvert=2$. 
    Indeed, 
    the field $K$ can be written as 
    $\Q(\beta)$ with $\beta \in \C \backslash \R$. 
    Since any automorphism $\rho \in \Aut(K/\Q)$
    must map $\beta$ to $\beta$ or $\conj{\beta}$, 
    we have $\ord \rho=1$ or $2$. 
    As $\ord \rho$ divides $n=2s'+3$, 
    we have $\rho=\id$. 

    The technique to construct an irreducible polynomial 
    with a given signature is explained in \cite[Remark 1.1]{OT05}. 
    We again note that an OT manifold of type 
    $(s,1)$ is of simple type. 
    Moreover, 
    it is proved that 
    an OT manifold $X(K,U,T)$ of type $(s,t)$ 
    admits an LCK metric if and only if $t=1$ 
    (see \cite{DV23}, \cite{Kan25}). 
\end{example}

\bibliographystyle{amsalpha}
\bibliography{OT}
\end{document}